\documentclass[11pt]{amsart}
\usepackage{geometry} 
\usepackage{pgfplots}
\usepackage{psfrag}
\usepackage{amsmath}
\usepackage{calc}
\usepackage{systeme}
\usepackage{amsfonts}
\usepackage{amsthm}
\usepackage{algorithm2e}
\usepackage[applemac]{inputenc}
\usepackage{bbold}
\usepackage[numbers]{natbib}

\numberwithin{equation}{section}

\author{%
  Charles Bertucci$^1$, Jean-Michel Lasry$^2$, Pierre-Louis Lions$^{2,3}$}

\newtheorem{Theorem}{Theorem}[section]

\newtheorem{Rem}[Theorem]{Remark}
\newtheorem{Def}[Theorem]{Definition}
\newtheorem{Prop}[Theorem]{Proposition}
\newtheorem{Cor}[Theorem]{Corollary}

\usepackage[foot]{amsaddr}

\newcommand{\be}{\begin{equation}}
\newcommand{\ee}{\end{equation}}

\newcommand{\mptd}{\mathcal{P}(\mathbb{T}^d)}

\newcommand{\R}{\mathbb{R}}
\newcommand{\T}{\mathbb{T}}

\title{On Lipschitz solutions of mean field games master equations}
\thanks{$^1$ : CMAP, Ecole Polytechnique, UMR 7641, 91120 Palaiseau, France\\
$^2$: CEREMADE, Universit\'e Paris Dauphine, UMR 7534, 75016 Paris, France\\
$^3$: Coll\`ege de France, 11 place Marcelin Berthelot, 75005 Paris, France
}
\date{} 

\begin{document}
\maketitle
\begin{abstract}
We develop a theory of existence and uniqueness of solutions of MFG master equations when the initial condition is Lipschitz continuous. Namely, we show that as long as the solution of the master equation is Lipschitz continuous in space, it is uniquely defined. Because we do not impose any structural assumptions, such as monotonicity for instance, there is a maximal time of existence for the notion of solution we provide. We analyze three cases: the case of a finite state space, the case of master equation set on a Hilbert space, and finally on the set of probability measures, all in cases involving common noises. In the last case, the Lipschitz continuity we refer to is on the gradient of the value function with respect to the state variable of the player.
\end{abstract}
\tableofcontents

\section*{Introduction}
In this paper, we explain why there can be at most one "sufficiently smooth" solution of mean field games (MFG in short) master equations, in three different settings. We first consider the case of a MFG master equation associated to a finite state space model. In such a case, the master equation is a finite dimensional partial differential equation PDE (in short). We then present analogue results in the so-called Hilbertian approach, in which the master equation is a PDE set on a general Hilbert space; as well as a master equation associated to a MFG with a continuous state space. Of course the concept of a "sufficiently" smooth solution shall be detailed later on in the paper, and depends on the setting of the equation, but it turns out that Lipschitz continuity in space/measure is an important notion for the well posedness.\\

The mathematical study of MFG master equations have been initiated by the second and third authors. These equations are PDE, which can be written in either finite or infinite dimensional sets. They are an essential tool in the theory of MFG which are dynamic games involving non-atomic agents \citep{lasry2007mean}. The main advantage of master equations is that they allow to model a wide variety of games, namely ones involving so-called common noises. We refer the reader to \citep{lions2007cours,cardaliaguet2019master,carmona2018probabilistic,carmona2018probabilistic2,bertucci2021monotone,bertucci2021monotone2} for more details on the study of MFG master equations.\\

The main mathematical difficulty arising in the study of MFG master equations is that, in general, such equations create shocks, or discontinuities. For instance, one of the simplest and most famous PDE which can be a non trivial master equation is the Burger's equation. The so-called monotone regime has been identified in \citep{lasry2007mean}. It prevents shocks and allows to propagate some regularity for MFG master equations \citep{lions2007cours,cardaliaguet2019master,bertucci2021monotone2,bertucci2019some,bertucci2021master}. Hence in this monotone regime, one can study classical solutions of the problem on time intervals of arbitrary length. Moreover, a weaker notion of monotone solution has been introduced in \citep{bertucci2021monotone,bertucci2021monotone2} and studied in \citep{cardaliaguet2022monotone}. It allows to deal with non-classical solutions in the monotone regime. Let us also mention the work \citep{mou} which also deals with the monotone regime. Other regimes of well-posedness exist, such as the so called displacement monotone one, see for instance \citep{gangbo,zhang}. However no proper notion of weak solution exists in general.\\

Several results of well-posedness on short time horizons have been established, like for instance \citep{lions2007cours,carmona2018probabilistic2,cardaliaguet2022splitting,ambrose2021well}. In these results, the authors prove that if the time horizon is sufficiently small, then there always exists a unique classical solution of the master equation. This approach is close to the the one we adopt here. In some sense our results state that this idea can be pushed further, namely by using a weaker notion of solution. Moreover, we can characterize this maximal time of well-posedness in terms of the Lipschitz regularity of the solution.\\

In this paper, we establish the fact that, starting from a sufficiently regular initial condition, there exists exactly one solution of the master equation in a certain regularity class up to a certain time, which depends on the initial condition. We believe this fact to be quite general and to rely mostly on the form of the master equations. Even though it seems unfeasible to treat all possible master equations at once, we give three examples of settings in which this phenomenon happens to convince the reader of the generality of this approach. We end the paper by discussing the implications that such results may have as well as by giving some directions of extensions of this work.\\

Finally, let us insist on the fact that the approach we shall propose is clearly in the spirit of the Cauchy-Lispchitz (or Picard-Lindel\"of) theory of ordinary differential equations.\\

The rest of the paper is organized as follows. Section \ref{sec:finite} presents our approach on the more simple case of a finite state space. Section \ref{sec:hilbert} generalizes this approach to master equations set on Hilbert spaces and Section \ref{sec:c} treats the case of a continuous state space. Finally we have gathered additional remarks and perspectives in Section \ref{sec:add}.

\section{The case of a finite state space}\label{sec:finite}
For MFG with a finite ($d\geq 1$) number of states, the master equation generally takes the following form 
\be\label{mfgd}
\partial_t U(t,x) + \langle F(x,U),\nabla_q\rangle U(t,x) + \lambda (U(t,x) - (DS)^* U(t,Sx)) = G(x,U) \text{ in } (0,\infty) \times \mathcal{O},
\ee
\be\label{initd}
U(0,x) = U_0(x) \text{ in } \mathcal{O},
\ee
where $U : [0,\infty) \times \mathcal{O} \to \mathbb{R}^d$ and $\mathcal{O}$ is bounded domain of $\mathbb{R}^d$. The terms $F: \mathcal{O}\times \mathbb{R}^d\to \mathbb{R}^d$ and $G: \mathcal{O}\times \mathbb{R}^d\to \mathbb{R}^d$ model strategic interactions between the players, $U_0$ is an initial condition and the terms involving $\lambda> 0$ and $S : \mathcal{O} \to \mathcal{O}$ model common noise in the game. We refer to \citep{lions2007cours,bertucci2019some} for more details on such models. For the rest of this section, $\langle \cdot,\cdot \rangle$ denotes the euclidean scalar product of $\mathbb{R}^d$.
\begin{Rem}
In this setting, for most of the model studied in the literature, the variable $x$ stands for the repartition of players. This choice of variable may seem strange since $x$ denotes the state of a single player in lots of models. However, because in several key example of master equation this variable $x$ does not stand for a repartition of players \citep{bertucci2020mean,achdou2022class}, we made this choice. A more general statement is that quite often master equations can be of some importance in themselves, even though there is no proper underlying MFG.
\end{Rem}

Concerning the boundary conditions, we assume that $\mathcal{O}$ is smooth and that
\be\label{stabO}
\forall x \in \partial \mathcal{O}, \forall p \in \mathbb{R}^d, \langle \eta(x), F(x,p)\rangle \geq 0,
\ee
which ensures that no additional boundary condition is needed. In the previous assumption, $\eta(x)$ stands for the normal vector to $\partial \mathcal{O}$, at the point $x$, pointing outward. \begin{Rem}The smoothness of $\mathcal{O}$ is by no means important here. All the following can be extended to non smooth $\mathcal{O}$ (which is often the case in applications). However, we make this assumption to simplify the following discussion.
\end{Rem}
Accordingly, we also assume that $S$ is such that $S(\mathcal{O}) \subset \mathcal{O}$.

\subsection{The notion of Lipschitz solution for master equations}
The main remark at the origin of the notion of solution we propose here is the following: If we know the map $(t,x) \to (F(x,U(t,x)),G(x,U(t,x)))$, then computing the solution of \eqref{mfgd} reduces to solving linear transport equation. Hence, a solution of \eqref{mfgd} can be expressed a fixed point of the composition of solving the linear transport equation and evaluating $F$ and $G$.\\

We now give a notion of solutions of the linear transport equation at interest, which is by now standard, and refer to \citep{lions2022cours} for a presentation of recent developments on transport equations. Consider the following transport equation of solution $V : [0,\infty) \times \mathcal{O}\to \R^d$
\be\label{ld}
\partial_t V(t,x) - \langle B(t,x),\nabla_x\rangle V(t,x) = A(t,x) \text{ in } (0,\infty)\times \mathcal{O},
\ee
with initial condition
\be
V|_{t = 0}(x) = U_0(x) = \text{ in } \mathcal{O}.
\ee
In the previous, $A, B : (0,\infty) \times \mathcal{O}\to \R^d$ and $U_0 : \mathcal{O} \to \R^d$. The linear transport equation \eqref{ld} can be solved by using the method of characteristics that we now recall briefly. Let us assume that $B$ is, uniformly in $t$, Lipschitz continuous in $x$. This implies that the ordinary differential equation (ODE) 
\be\label{oded}
\frac{d}{ds} x(s) = B(t-s,x(s)),
\ee
has a unique solution given any initial condition $x\in \mathcal{O}$ and time interval $[-t,t]$. Furthermore, assume that $B$ satisfies \eqref{stabO} so that $\mathcal{O}$ is invariant by \eqref{oded}, i.e. $x(s) \in \mathcal{O}$ for all $s\in[0,t]$. The method of characteristics states that a classical solution $V$ of \eqref{ld} satisfies 
\be\label{chard}
V(t,x) = \int_0^t A(t-s,x(s))ds + U_0(x(t)).
\ee
Here, we shall use this relation as the notion of solution of \eqref{ld}.
\begin{Def}
Take a final time $T$, a time dependent vector field $B : [0,T] \times \mathcal{O} \to \R^d$ satisfying \eqref{stabO} as well as being, uniformly in $t$, Lipschitz continuous in $x \in \mathcal{O}$. For $A$ and $U_0$ continuous, the solution of \eqref{ld} is the function defined for every $(t,x) \in [0,T]\times \mathcal{O}$ by \eqref{chard}. In this case, we note
\be
V = \Psi(T,A,B,U_0).
\ee
\end{Def}
\begin{Rem}
The previous definition indeed defines a mapping $\Psi$ because we can consider a solution of \eqref{oded} from any initial condition and arbitrary time length.
\end{Rem}

We then provide the following definition of a Lipschitz solution of the master equation \eqref{mfgd}.
\begin{Def}\label{deflipd}
A Lipschitz solution $U$ of \eqref{mfgd}, on the time interval $[0,T)$, is a function such that 
\begin{itemize}
\item $U$ is Lipschitz in $x \in \mathcal{O}$, uniformly in $t \in [0,\alpha]$ for any $\alpha < T$.
\item For any $t < T$
\be
U = \Psi\bigg(t,G(\cdot,U) - \lambda(U - (DS)^*U\circ S),-F(\cdot,U),U_0\bigg).
\ee
\end{itemize}
\end{Def}
This, quite classical, definition of a solution of a non-linear PDE has the advantage that it allows us to define a concept of solution for merely Lipschitz (in the space variable) functions. The following result is immediate.

\begin{Prop}
A classical solution of \eqref{mfgd} is a Lipschitz solution of \eqref{mfgd} in the sense of Definition \ref{deflipd} on any time interval. Moreover a Lipschitz solution is a classical solution if it is smooth.
\end{Prop}
\begin{proof}
\textit{First part of the claim.}The proof of this result is a simple computation. Considering a classical solution $U$. Take the time derivative of both $U$ and $\Psi(t,G(\cdot,U) - \lambda(U - (DT)^*U\circ T),F(\cdot,U),U_0)$. It is immediate to verify that they are actually the same, hence that $U$ is also a Lipschitz solution of \eqref{mfgd}.\\

\textit{Second part of the claim.} Consider a smooth Lipschitz solution of \eqref{mfgd} on $[0,T)$. Take $t > 0, x\in \mathcal{O}$ and denote by $(x(s))_{s \in [0,t]}$ the solution of 
\be
\frac{d x(s)}{ds} = -F(x(s),U(t-s,x(s)))
\ee
 with initial condition $x$. Observe now that, for $dt > 0$ sufficiently small
\be
U(t,x) - U(t-dt, x(dt)) = \int_0^{dt}G(x(s),U(t-s,x(s))) - \lambda(U(t-s,x(s)) - (DS(x(s)))^*U(t-s,Sx(s))) ds.
\ee
Hence dividing the previous relation by $dt$ and letting $dt \to 0$, we recover that $U$ indeed solves \eqref{mfgd}.
\end{proof}

\subsection{The main result on Lipschitz solutions}
In the spirit of the usual theory of ODE, we are able to prove the following result for such Lipschitz solutions.

\begin{Theorem}\label{thm:lipd}
Assume that $F$ and $G$ are locally Lipschitz functions, that $S$ is smooth and consider a Lipschitz initial condition $U_0 : \mathcal{O} \to \R^d$.
\begin{itemize}
\item There always exists a time $T > 0$ such that there exists a unique solution of \eqref{mfgd} in the sense of Definition \ref{deflipd} on $[0,T)$.
\item Moreover, there exists a maximal time $T^c \in [0,\infty]$ and a solution $U$ associated to $T^c$ such that, for any solution $V$ of the problem on an interval $[0,T]$: we have that $T \leq T^c$ and the restriction of $U$ to $[0,T)$ is equal to $V$.
\item If $T^c< \infty$, then $\|D_xU(t,\cdot)\|_{\infty} \to \infty$ as $t \to T^c$.
\end{itemize}
\end{Theorem}
\begin{Rem}
Let us insist on the fact that $\mathcal{O}$ is assumed to be bounded here, and hence that $U_0$ is bounded and that, for $ C>0$, $F$ and $G$ are Lipschitz continuous in $x \in \mathcal{O}$, uniformly in $|p|\leq C$. Cases in which the domain of the equation is unbounded are treated in the next section.
\end{Rem}
\begin{proof}
Consider $T > 0$ and introduce the mapping $\Phi$ defined by
\be
\begin{aligned}
\Phi : &E_{\infty} \to E_{\infty},\\
&U \to \Psi\bigg(T,G(\cdot,U) - \lambda(U - (DS)^*U\circ S),-F(\cdot,U),U_0\bigg),
\end{aligned}
\ee
where 
\be
E_C := \{U : [0,T)\times \mathcal{O} \to \R^d, \|U\|_{\infty} < C, \|D_x U\|_{\infty} < C\}.
\ee
\textbf{Step $1$: $\Phi$ is well defined.} First we want to show that $\Phi$ is well defined, i.e. that it takes its values in $E_{\infty}$. Take $U \in E_{\infty}$. The fact that $\Phi(U)$ is bounded is a direct consequence of \eqref{chard}. We now show that $\|D_x\Phi(U)\|_{\infty}< \infty$. In fact, we are going to show that as soon as $A,B$ and $U_0$ are Lipschitz in the space variable $x$, uniformly in $[0,T)$, then $\|D_x\Psi(T,A,B,U_0)\|_{\infty} < \infty$.

Let us denote by $\xi(t,s,x)$ the flow of the ODE \eqref{oded}. That is, given $t> 0$, an initial condition $x \in \mathcal{O}$, $\xi(t,s,x) = x(s)$ where $(x(s))_{s \geq 0}$ is the unique solution of \eqref{oded}. 

Differentiating \eqref{chard} with respect to $x$ yields for $t < T,x \in \mathcal{O}$
\be
D_x \Psi(T,A,B,U_0)(t,x) = \int_0^t D_x A(s,\xi(t,s,x))D_x\xi(t,s,x) ds + D_x U_0(x(t))D_x \xi(t,t,x).
\ee
Let us recall that, since $B$ is uniformly Lipschitz, the flow $\xi$ is Lipschitz in $x$, uniformly in $t,s$. Then the bound on $\|D_x \Psi(T,A,B,U_0)\|_{\infty}$ follows, from which we deduce that $\Phi(U) \in E_{\infty}$. Moreover, if $T$ is chosen sufficiently small, then there exists $C$ such that
\be
\sup_{t \leq t_1} \|D_xU\|_{\infty} \leq C \Rightarrow \sup_{t \leq t_1}\|D_x \Phi(U)\|_{\infty} \leq C.
\ee
Hence, if $T$ is sufficiently small and $C$ sufficiently large, $\Phi$ maps $E_C$ into itself.\\

\textbf{Step 2: $\Phi$ is a contraction.} We now show that $\Phi$ is a contraction on $E_C$ for the $\|\cdot\|_{\infty}$ norm, if $T$ is chosen small enough and $C$ is chosen large enough. From \eqref{chard}, we deduce that, for any $U,V \in E$, $t< T, x \in \mathcal{O}$
\be\label{eqcontract}
\begin{aligned}
(\Phi(U)& - \Phi(V))(t,x) = \int_0^t G(\xi_1(t,s,x),U(s,\xi_1(t,s,x))) - G(\xi_2(t,s,x),V(s,\xi_2(t,s,x)))ds \\
- \lambda &\int_0^t (DS(\xi_1(t,s,x)))^*U(t-s,T(\xi_1(s-t,x))) -  (DS(\xi_2(t,s,x)))^*V(t-s,S(\xi_2(t,s,x)))ds\\
&\quad \quad \quad + \lambda\int_0^tV(t-s,\xi_2(t,s,x)) - U(t-s,\xi_1(t,s,x)) ds,\\
&\quad \quad \quad + U_0(\xi_1(t,s,x)) - U_0(\xi_2(t,s,x)),
\end{aligned}
\ee
where $\xi_1$ and $\xi_2$ are the flows associated to respectively $U$ and $V$, that is $\xi_1(t,\cdot,\cdot)$ is the flow of the ODE
\be
\frac{d}{ds} x(s) = - F(x(s),U(t-s,x(s))),
\ee
while $\xi_2(t,\cdot,\cdot)$ is the flow of the same ODE when $U$ is replaced by $V$. To deduce the contraction property from \eqref{eqcontract}, it suffices to show that 
\be\label{eq15}
|\xi_1(t,s,x) -\xi_2(t,s,x)| \leq t C \|U-V\|_{\infty}.
\ee
Indeed if \eqref{eq15} holds, then we can bound the right hand side of \eqref{eqcontract} with
\be
\begin{aligned}
&t(C_{G}C \|U-V\|_{\infty} + C_G( \|U-V\|_{\infty} + \|D_xU\|_{\infty}C\|U-V\|_{\infty}))+\\
&+ \lambda t(\|D^2S\|_{\infty}\|U\|_{\infty}C \|U-V\|_{\infty} + \|DS\|_{\infty}^2( \|U-V\|_{\infty} + \|D_xU\|_{\infty}C\|U-V\|_{\infty}) )\\
& + \lambda t( \|U-V\|_{\infty} + \|D_xU\|_{\infty}C\|U-V\|_{\infty}) + \|D_xU_0\|_{\infty}tC\|U-V\|_{\infty},
\end{aligned}
\ee
where $C_G$ is the Lipschitz constant of $G$ on $\{(x,p) | x \in \mathcal{O}, |p|\leq C\}$. Hence, we deduce that for any $C > 0$ such that $\Phi(E_{C}) \subset \Phi(E_{C})$, there exists $T > 0$ such that $\Phi$ is a contraction on $E_{C}$ for the $\|\cdot\|_{\infty}$ norm. It then remains to show \eqref{eq15}. Let us compute, for $U,V \in E_C$
\be
\begin{aligned}
\left|\frac{d}{ds}(\xi_1(t,s,x) - \xi_2(t,s,x))\right| = &\left|F(\xi_1(t,s,x),U(t-s,\xi_1(t,s,x))) - F(\xi_2(t,s,x),V(t-s,\xi_2(t,s,x)))\right|\\
\leq& | F(\xi_1(t,s,x),U(t-s,\xi_1(t,s,x))) - F(\xi_1(t,s,x),V(t-s,\xi_1(t,s,x)))|\\
& + | F(\xi_1(t,s,x),V(t-s,\xi_1(t,s,x))) - F(\xi_1(t,s,x),V(t-s,\xi_2(t,s,x)))|\\
& + | F(\xi_1(t,s,x),V(t-s,\xi_2(t,s,x))) - F(\xi_2(t,s,x),V(t-s,\xi_2(t,s,x)))|\\
\leq & C_F \|U-V\|_{\infty} + C_F\|D_xV\|_{\infty}|\xi_1(t,s,x) - \xi_2(t,s,x)|\\
& + C_F|\xi_1(t,s,x) - \xi_2(t,s,x)|.
\end{aligned}
\ee
Defining $I(s) = |\xi_1(t,s,x) - \xi_2(t,s,x)|$, we recognize that it satisfies an inequality of the form
\be
\frac{d}{ds}I(s) \leq C( 1 + I(s)).
\ee
Hence, since $I(0) = 0$, we obtain from Gr\"onwall's Lemma that $I(s) \leq e^{Cs} -1$, from which \eqref{eq15} follows.\\

\textbf{Step $3$: Existence of a fixed point.} Let us now remark that for $C > 0$, there exists $L > 0$ such that for any $U \in E_C, \|\partial_t U\|_{\infty}\leq L$. Hence, $\Phi(E_C)$ is contained in subset of uniformly Lipschitz and bounded functions on $[0,T)\times \mathcal{O}$.\\

We now take $C>0$ large enough so that $\Phi(E_C) \subset E_C$ and $T>0$ small enough so that $\Phi$ is a contraction on $(E_C,\|\cdot\|_{\infty})$. Frollowing the classical proof of Picard's fixed point theorem, we deduce, using \textbf{Step $2$}, that for any $U \in E_C$, the sequence $(\Phi^n(U))_{n \geq 0}$ is a Cauchy sequence in $(E_C,\|\cdot\|_{\infty})$. Hence it converges, thanks to Ascoli-Arzela Theorem, to a limit $U^* \in E_C$. From the continuity of $\Phi$, we deduce that $U^*$ is indeed a fixed point of $\Phi$.\\

\textbf{Step $4$: Uniqueness of solutions and critical time of existence.} To prove the rest of the claim, let us consider
\be
T^c = \sup\{T > 0 | \exists \text{ a Lipschitz solution of \eqref{mfgd} on } [0,T)\},
\ee
and sequences $(T_n)_{n \geq 0}, (U_n)_{n \geq 0}$ such that for all $n \geq 0$, $U_n$ is a Lipschitz solution of \eqref{mfgd} on $[0,T_n)$ and $T_n \uparrow_{n \to \infty}T^c$. For $n \leq m$, consider $T_* = \inf \{t \in [0,T_n),\exists x \in \mathcal{O}, U_n(t,x)\ne U_m(t,x)\}$. If $T_* > -\infty$, then by using \textbf{Step $2$} on the initial condition $U_0(x) = U_n(t-\epsilon,x)$ for $\epsilon > 0$ sufficiently small, we arrive at a contradiction. Hence we deduce that all the Lipschitz solutions coincides up to time $T^c$.

Let us now assume that $T^c < \infty$ and that there exists $C > 0$ such that $\|D_xU_n\|_{\infty} \leq C$ for all $n \geq 0$. Consider $\delta > 0$. By applying \textbf{Step $3$} to the initial condition $U_n(T^c-\delta,x)$, for $n$ such that $T_n \geq T^c - \delta$, we deduce that there exists a time of existence of a solution $\epsilon > 0$, which is bounded from below by a constant which depends only on $C$, thanks to the computation of the previous part of the proof. Hence, choosing $\delta > 0$ sufficiently small, we arrive at a contradiction since this allows to extend the solution $U_n$ on $[0,T_n - \delta + \epsilon)$, while being a solution of \eqref{mfgd} on this time interval. Hence, we necessary have that $\lim_{t \to T^c}\|D_x U(t,\cdot)\|_{\infty} = + \infty$ (in the case $T^c < \infty$).

\end{proof}
\begin{Rem}
Obviously the previous proof is very much in the spirit of the standard theory of ODE. Maybe the main difference here is that we show that $\Phi$ is a contraction for the $\|\cdot\|_{\infty}$ norm while we need to verify that it is defined from a space of Lipschitz functions into itself.
\end{Rem}

\subsection{Comparison with the usual notion of characteristics for MFG}
Let us recall that, in the MFG community, the notion of characteristics associated to \eqref{mfgd} is in general a forward backward system. In this setting, this system takes the form
\be
\begin{cases}
\frac{d}{dt} Y(t) = G(y(t),Y(t)) \text{ for } t\in (0,T),\\
-\frac{d}{dt}y(t) = -F(y(t),Y(t)) \text{ for } t \in (0,T),
\end{cases}
\ee
which is often associated with boundary conditions of the form $Y(0) = U_0(y(0)), y(T) = x_0$, where $x_0 \in \mathcal{O}$, and $T > 0$.\\

In our approach, we do not need such complex forward-backward system of characteristics, since, in particular, we can deal with only the forward equation \eqref{oded}. We believe that this makes our approach easier to work with than most of the existing literature on the construction of solution of the master equation in short time intervals.

\subsection{Time regularity of Lipschitz solutions}
We now explain why Lipschitz solutions of \eqref{mfgd} are necessary locally Lipschitz in $t$. Even though this was present in the proof, we present the following argument as it is more intrinsic. To simplify the notation, we consider the case $\lambda = 0$, although this does not bear any importance. This fact is quite simple and maybe the best understanding is through the fact, once a bound exists on $D_xU$, we can read a bound on $\partial_t U$ on \eqref{mfgd}. A more rigorous approach consists in making the computation, for $U$ a Lipschitz solution of \eqref{mfgd} and $t \geq s > 0, x \in \mathcal{O}$, and $x(\cdot)$ the solution of \eqref{chard} with initial condition $x$
\be
\begin{aligned}
|U(t,x) - U(s,x)| &\leq |U(t,x) - U(s,x(t-s))| + |U(s,x(t-s)) - U(s,x)|\\
& \leq \left|\int_s^tG(x(s'), U(t-s',x(s')))ds'\right| + C|x(t-s) - x|\\
& \leq \left|\int_s^tG(x(s'), U(t-s',x(s')))ds'\right| + C\left|\int_0^{t-s}F(x(s'), U(t-s',x(s')))ds'\right|\\
&\leq C |t-s|.
\end{aligned}
\ee
From this computation, we obtain the
\begin{Prop}
Under the assumptions of Theorem \ref{thm:lipd}, consider the Lipschitz solution $U$ of \eqref{mfgd} on the maximal time of existence $T^c$. Then, for any $T < T^c$, $U$ is Lipschitz on $[0,T]\times \mathcal{O}$.
\end{Prop}
\begin{Rem}
To avoid repeating the same argument multiple times, we shall not present this result for the next cases, although they actually hold. Maybe the only point of difference is that in general, the Lipschitz regularity in time is only local in space, which does not appear since $\mathcal{O}$ is bounded.
\end{Rem}

\subsection{A strong-weak like uniqueness result}
Even though there does not exist a general notion of weak solution of MFG master equations, we explain why the existence of a Lipschitz solution may yield uniqueness of solutions in a wider class of notion of solutions. To present this idea, we show that any limit $V$ of a sequence of smooth functions $(V_{\epsilon})_{\epsilon > 0}$ which are almost solution of \eqref{mfgd}, is actually equal to the Lipschitz solution of \eqref{mfgd} $U$, when $U$ exists of course. Once again, we are in this section in the case $\lambda = 0$, to simplify notation. We can prove the following.

\begin{Prop}
Consider a Lipschitz solution $U$ of \eqref{mfgd} on the time interval $[0,T^c)$. Assume that, for any $T \in (0,T^c)$, there exists $C >0$ such that for any $\epsilon > 0$, any classical solution $V_{\epsilon}$ of
\be
\begin{aligned}
\left|\partial_t V_{\epsilon} + \langle F(x,V_{\epsilon}),\nabla_x\rangle V_{\epsilon} - G(x,V_{\epsilon})\right| \leq \epsilon \text{ in } [0,T)\times \mathcal{O},\\
|V_{\epsilon}(0,x) - U_0(x)| \leq \epsilon \text{ in } \mathcal{O}.
\end{aligned}
\ee
Then the following holds for some constant $C$ depending only on $\epsilon, T, U_0, F$ and $G$
\be
\sup_{t \leq T, x \in \mathcal{O}}|U(t,x) - V_{\epsilon}(t,x)| \leq C \epsilon.
\ee
\end{Prop}
\begin{proof}
To simplify notation we omit the index $\epsilon$ on $V$. Let us consider $t > 0$ and $i \in \{1,...,d\}, x^* \in \mathcal{O}$ such that 
\be
U^i(t,x^*) - V^i(t,x^*) = \max_{j,x}\{U^j(t,x) - V^j(t,x)\}
\ee Consider $\kappa > 0$ and the solution $x(\cdot)$ of the ODE
\be\label{ode12}
\frac{d}{ds}x(s) =- F(x(s),U(t+ \kappa -s,x(s))),
\ee
with initial condition $x^*$ and $y(\cdot)$ the solution of the ODE
\be\label{ode22}
\frac{d}{ds}y(s) = -F(y(s),V(t+ \kappa -s,y(s))),
\ee
with also initial condition $x^*$. Let us now compute
\be\label{eq2er}
\begin{aligned}
U^i&(t+ \kappa, x^*) - V^i(t+ \kappa, x^*) - U^i(t,x^*) + V^i(t,x^*)=\\
&=  U^i(t+ \kappa, x^*) - U^i(t,x(\kappa))  - V^i(t+ \kappa, x^*) + V^i(t,y(\kappa))\\
& \quad + U^i(t,x(\kappa))- U^i(t,x^*) + V^i(t,x^*) - V^i(t,y(\kappa))\\
& \leq \int_0^{\kappa} G^i(s,U(t+\kappa -s,x(s))) - G^i(s,V(t + \kappa -s, y(s))) + \epsilon ds\\
&+ U^i(t,x(\kappa)) - U^i(t,y(\kappa)).
\end{aligned}
\ee
Let us now remark that, since $U$ is uniformly Lipschitz, we deduce that
\be
|U^i(t,x(\kappa)) - U^i(t,y(\kappa))| \leq C |x(\kappa) - y(\kappa)|.
\ee
From \eqref{ode12} and \eqref{ode22} we finally deduce that
\be
|U^i(t,x(\kappa)) - U^i(t,y(\kappa))| \leq C \kappa \|U-V\|_{\infty}.
\ee
Hence, dividing by $\kappa$ and taking the limit $\kappa \to 0$ in \eqref{eq2er}, we obtain
\be
\frac{d}{dt}(U^i(t, x^*) - V^i(t, x^*)) \leq C \|U(t,\cdot) - V(t,\cdot) \|_{\infty} + \epsilon.
\ee
From this, we obtain that
\be
\frac{d}{dt}\|U(t,\cdot) - V(t,\cdot)\|_{\infty} \leq C  \|U(t,\cdot) - V(t,\cdot) \|_{\infty} + \epsilon.
\ee
Hence, we deduce from Gr\"onwall's Lemma that
\be\label{lasteqprop}
\|U(t,\cdot) - V(t,\cdot)\|_{\infty} + \epsilon \leq 2\epsilon e^{Ct},
\ee
from which the Proposition immediately follows.
\end{proof}
\begin{Rem}
From the previous result we can indeed deduce that the Lipschitz solutions attracts the limits of approximations of \eqref{mfgd}, since the constant $C$ in \eqref{lasteqprop} does not depend on $V_{\epsilon}$.
\end{Rem}

 \section{The Hilbertian case}\label{sec:hilbert}
In the so-called Hilbertian approach, introduced by the third author, the typical form of a MFG master equation with common noise, is 
\be\label{mfgHn}
\partial_t U(t,X) -\sum_{i=1}^{\infty} \lambda_i\partial_{ii}U(t,X)+ \langle F(X,U), \nabla\rangle U(t,X) =G(X,U) \text{ in } (0,\infty)\times H,
\ee
 where $(H, \langle \cdot,\cdot \rangle)$ is a real separable Hilbert space, with an orthonormal family $(e_i)_{i \geq 1}$  and $F,G : H\times H \to H$. In the previous, the term $ \langle F(x,U), \nabla\rangle U $ is understood in the sense that we are taking the Gateaux derivative of $U$ in the direction $F(x,U)$ and $\partial_{ii}$ refers to the second order derivative with respect to $e_i$.
 
 The master equation \eqref{mfgHn} is related to more "classical" master equations set on the space of probability measures. A typical example of the latter is written below in equation \eqref{mfgcc}. The two approaches are concerned with the same model when:
 \be
 \begin{aligned}
 &H = L^2(\Omega, \R^d), B = D_p H\\
 &\sigma = \sigma' = 0, \lambda_i = \sigma_0 \mathbb{1}_{i \leq d}, F(X,U) = D_pH(X,U),\\
  &G(X,U) =  - D_x H(X,U,\mathcal{L}(X)),
  \end{aligned}
  \ee
  where $(\Omega, \mathcal{A},\mathbb{P})$ is a standard probability space and $\mathcal{L}(X)$ denotes the law of the random variable $X$, and when the first $d$ components of $(e_i)_{i \geq 1}$ are the component of the canonical basis of $\R^d$. In this case, we expect that given classical solutions $U$ and $\mathcal{U}$ of respectively \eqref{mfgHn} and \eqref{mfgcc}, the following holds for all $X \in L^2(\Omega, \R^d)$,
  \be\label{lift}
  \nabla_x \mathcal{U}(t,X,\mathcal{L}(X)) = U(t,X).
  \ee
  
 \subsection{The deterministic case}
 As the theory of ODE in Hilbert spaces does not raise any particular difficulty, we can easily adapt the results of the previous part to the so-called deterministic equation (i.e. the case $\forall i,\lambda_i = 0$)
 \be\label{mfgH}
 \partial_t U + \langle F(x,U), \nabla\rangle U = G(x,U) \text{ in } (0,\infty)\times H,
 \ee
 \be\label{initH}
 U(0,x) = U_0(x) \text{ in } H.
 \ee
 In tis context, there is also an associated linear transport equation. Given two vectors fields $A,B : [0,T)\times H \to H$, we say that $V:[0,T)\times H \to H$ is a solution of the linear transport equation
 \be
 \partial_t V - \langle B(t,x), \nabla\rangle V = A(t,x) \text{ in } (0,T)\times H
 \ee
 if $V$ satisfies for all $t \in [0,T), x \in H$
 \be\label{charh}
V(t,x) = \int_0^t A(t-s,\xi(t,s,x))ds + U_0(\xi(t,t,x)),
 \ee
 where $\xi(t,\cdot,\cdot)$ is the flow of the ODE
 \be
 \frac{d}{ds}x(s) = B(t-s,x(s)).
 \ee
 Note that, as in the finite state space case, the flow is well defined as soon as $B$ is Lipschitz in $x$, uniformly in $t$ for instance. Hence in this case, the formula \eqref{charh} makes sense as soon as $A$ and $U_0$ are continuous for instance. We also define the operator $\Psi$ with $\Psi(T,A,B,U_0)$ is the function given by \eqref{charh}. As in the previous section we can define a notion of Lipschitz solution and give a result of existence and uniqueness.
 \begin{Def}\label{defliph}
 A Lipschitz solution $U$ of \eqref{mfgH} on the time interval $[0,T)$ is a function such that 
\begin{itemize}
\item $U$ is Lipschitz in $x \in H$, uniformly for $t \in [0,\alpha]$ for any $\alpha < T$.
\item For any $t < T$
\be
U = \Psi\bigg(t,G(\cdot,U) ,-F(\cdot,U),U_0\bigg).
\ee
\end{itemize}
 \end{Def}
 \begin{Theorem}\label{thm:hd}
 Assume that $F$ and $G$ are Lipschitz functions. For any Lipschitz initial condition $U_0 : H \to H$:
 \begin{itemize}
 \item There always exists a time $T > 0$ such that there exists a unique solution of \eqref{mfgd} in the sense of Definition \ref{defliph}. 
 \item There exists a maximal time $T^c \in [0,\infty]$ and a solution $U$ associated to $T^c$ such that, for any solution $V$ of the problem on an interval $[0,T]$: we have that $T \leq T^c$ and the restriction of $U$ to $[0,T)$ is equal to $V$.
 \item If the maximal time $T^c$ is such that $T^c< \infty$, then $\|D_x U(t)\|_{\infty} \to \infty$ as $t\to T^c$.
 \end{itemize}
 \end{Theorem}
\begin{proof}
This proof is very similar to the one in the finite state space case, hence we only the detail the main difference which is the fact that we consider an unbounded initial condition here. Since we assumed that $F$ and $G$ are globally Lipschitz, the only point where we need some uniform estimates on Lipschitz solutions is when we consider the set $E_C$ on which we want to use some fixed points results. In this setting, it is natural to consider the set defined by
\be\label{defEC}
E_C = \left\{U : [0,T)\times H \to U, \sup_{R > 0}\sup_{t \in [0,T), |x|\leq R}R^{-1}|U(t,x)| < C, \|D_XU\|_{\infty} < C \right\}.
\ee
The argument of the previous proof can be carried on on those sets simply by remarking that now, all the functional convergence shall be locally uniformly in $x \in H$.
\end{proof}
As an application of the previous result, let us consider the case of the master equation
\be
\partial_t U + \langle U, \nabla_x\rangle U = 0 \text{ in } (0,\infty)\times H,
\ee
with initial condition $U_0(x) = A(x)$ for some linear operator $A$ such that $A^* = A$. Remark that the solution of this master equation is simply given by $U(t,x) = A(t)x$, where $(A(t))_{t \geq 0}$ is the solution of 
\be
\frac{d}{dt}A(t) + A(t)^2 = 0.
\ee
Hence we are here in the situation $T^c < \infty$ except in the cases in which $A \geq 0$, which corresponds to the monotone regime which is known to propagate Lipschitz regularity for MFG master equation.

\subsection{The case of common noise}
We now turn to master equations involving a common noise, which is slightly more involved than the previous one. Moreover it will help us to understand how to extend this mathematical analysis to the case of master equations set on the space of probability measures. Recall that we are interested in the master equation 
\be\label{mfgHn}
\partial_t U(t,x) - \sum_{i=1}^{\infty} \lambda_i\partial_{ii}U(t,x)+ \langle F(x,U), \nabla\rangle U(t,x) =G(x,U) \text{ in } (0,\infty)\times H.
\ee
The presence of second order terms in this master equation imposes to use stochastic characteristics of the associated linear transport equation. Indeed, consider the linear (possibly degenerate) parabolic equation
 \be
 \partial_t V - \sum_{i = 1}^{\infty}\lambda_i\partial_{ii}V - \langle B(t,x),\nabla\rangle V = A(t,x) \text{ in } (0,\infty) \times H,
 \ee
 with initial condition
 \be
 V|_{t = 0}(x) = U_0(x) \text{ in } H,
 \ee
 where $A,B: [0,\infty)\times H \to H$ and $U_0:H \to H$. Let us consider, for $t > 0$ the stochastic differential equation (SDE in short) in $H$ that we write component wise on the family $(e_i)_{i \geq 1}$
 \be\label{sde}
 dX^i_s = B^i(t-s,X_s)ds + \sqrt{2 \lambda_i} dW^i_s,
 \ee
where $(W^i)_{1 \leq i }$ is a collection of independent real Brownian motions on $(\Omega, \mathcal{A},\mathbb{P})$. In the case in which $B$ is Lipschitz in $x \in H$, uniformly in $t$, the SDE is well defined, it even admits strong solutions. In this context, it is thus meaningful to introduce the following Feynman-Kac representation formula for $t< T, x \in H$
 \be
 V(t,x) = \mathbb{E}\left[\int_0^t A(t-s,X_s)ds + U_0(X_t) | X_0 = x\right],
 \ee
 where $(X_s)_{s \geq 0}$ is of course the solution of \eqref{sde}, conditioned here to take initial value $x$. The function $V$ defined by this formula is denoted once again by
 \be
 V = \Psi(T, A,B,U_0).
 \ee
 This remarks naturally leads us to the
 \begin{Def}\label{defhc}
 Given a time $T > 0$, a function $U: [0,T) \times H \to H$  is called a Lipschitz solution of \eqref{mfgHn} if
 \begin{itemize}
 \item $U$ is Lipschitz in $x \in H$, uniformly in $t \in [0,\alpha]$ for any $\alpha \in [0,T)$,
 \item For any $t < T$,
 \be
 U = \Psi(t,G(\cdot,U),-F(\cdot,U),U_0).
 \ee
 \end{itemize}
 \end{Def}
 As in the previous cases, we can establish the following Cauchy-Lipschitz like result.
 \begin{Theorem}\label{thmhc}
Assume that $F$ and $G$ are Lipschitz functions and that $\sum_{i = 1 }^{\infty}\lambda_i < \infty$. For any Lipschitz initial condition $U_0$:
\begin{itemize}
\item There always exists a time $T> 0$ such that there exists a unique solution of \eqref{mfgHn} in the sense of Definition \ref{defhc}. 
\item There exists a maximal time $T^c \in [0,\infty]$ and a solution $U$ associated to $T^c$ such that, for any solution $V$ of the problem on an interval $[0,T)$: we have that $T \leq T^c$ and the restriction of $U$ to $[0,T)$ is equal to $V$.
\item If $T^c< \infty$, then $\|D_x U(t)\|_{\infty}\to \infty$ as $t \to T^c$.
\end{itemize}
 \end{Theorem}
 \begin{proof}
 Consider $T > 0$, a bounded Lipschitz function $U_0$ and the function $\Phi$ defined by
 \be
 \Phi(U) = \Psi(T,G(\cdot,U),-F(\cdot,U),U_0).
 \ee
 Let us consider the set $E_C$ for $C > 0$, defined in \eqref{defEC}.\\
 
\textbf{Step 1: $\Phi$ is well defined.} Consider a function $U : [0,T)\times H \to H$ which is, uniformly in $t$, Lipschitz in $x\in H$. Then, uniformly in $t$, it is also the case for $x \to F(x,U(t,x))$ and $x \to G(x,U(t,x))$. This implies that, given $t > 0$ and an initial condition, the SDE
\be
dX^i_s = -F^i(X_s,U(t-s,X_s))ds + \sqrt{2 \lambda_i} dW^i_s,
\ee
is well defined. Denote by $(X^x_s)_{t \in [0,T)}$ the strong solution of the previous SDE with initial conditions $x\in H$.
Remark that, for any $x \in H$, $(X^x_s)_{s \geq 0}$ satisfies 
\be\label{eq:moment}
\begin{aligned}
\mathbb{E}[|X_s^x|] &\leq C \int_0^s\mathbb{E}[|X_{s'}|]ds' + \mathbb{E}\left[\left(\sum_i2 \lambda_i (W^i_s)^2\right)^{\frac12}\right]\\
&\leq C \int_0^s\mathbb{E}[|X_{s'}|]ds' + C(1 + \sum_{i}\lambda_i).
\end{aligned}
\ee
Hence $\sup_{s\geq 0} \mathbb{E}[|X_s^x|] < \infty$ and thus the function $\Phi$ is indeed well defined on the set of functions $U$ such that $\|D_x U\|_{\infty} < \infty$. It is valued in $E$.\\

\textbf{Step 2: $\Phi(E_C) \subset E_C$ for $C$ large enough and $T$ small enough.}
From \eqref{eq:moment}, we immediately deduce that if $T$ is small enough and $C$ is large enough, then $U \in E_C \Rightarrow \|\Phi(U)\|_{\infty} \leq C$. Let us now compute for $x,y \in H$ and $t > 0$
\be\label{eq:20}
\begin{aligned}
\left| \Phi(U)(t,x) - \Phi(U)(t,y) \right| =& \left |\mathbb{E}\left[ \int_0^t G(X^x_{s},U(t-s,X^x_{s})) - G(X^y_{s},U(t-s,X^y_{s}))ds \right]\right|\\
&+\left|\mathbb{E}\left[ U_0(X^x_t) - U_0(X^y_t)\right] \right|\\
\leq& C\mathbb{E}[|X^x_t- X_t^y|] + \mathbb{E}\left[ \int_0^t C|X^x_s - X^y_s| ds\right].
\end{aligned}
\ee
The following estimate holds almost surely
\be
d(X^x_t - X^y_t) \leq \|D_XF\|_{\infty}|X^x_t - X^y_t| + \|D_pF\|_{\infty}\|D_x U\|_{\infty}|X^x_t - X^y_t|
\ee
Hence we deduce from Gr\"onwall's Lemma and from \eqref{eq:20} that $D_x \Phi(U)$ is uniformly bounded and that, choosing $T$ sufficiently small, we can guarantee that there exists $C> 0$ such that
\be
\sup_{t \leq T} \|D_X U(t)\|_{\infty} \leq C \Rightarrow \sup_{t \leq T} \|D_X \Phi(U)(t)\|_{\infty} \leq C.
\ee 

\textbf{Step 3: $\Phi$ is a contraction if $T$ is small enough.} Consider $U,V \in E_C$, and let us compute for $t \leq T, x \in H$
\be\label{eq:21}
\begin{aligned}
&\left| \Phi(U)(t,x) - \Phi(V)(t,x) \right| =\\
&=  \left | \mathbb{E}\left[ U_0(X_t)- U_0(\tilde{X}_t) +\int_0^t G(X_s,U(t-s,X_s)) - G(\tilde{X}_s,V(t-s,\tilde{X}_s))ds\right] \right|,\\
\end{aligned}
\ee
where $(X_s)_{s \in [0,T]}$ is the strong solution of \eqref{sde} with initial condition $x$ and $(\tilde{X}_s)_{s \in [0,T]}$ is the strong solution of \eqref{sde} with initial condition $x$ when $U$ has been replaced by $V$. Hence, almost surely, we have the estimate
\be
d|X_s - \tilde{X}_s| \leq \|D_xF\|_{\infty}|X_s - \tilde{X}_s| + \|D_pF\|_{\infty}(\|U-V\|_{\infty} + \|D_xU\|_{\infty}|X_s- \tilde{X}_s|).
\ee
From this estimate, we immediately deduce from \eqref{eq:21} that $\Phi$ is indeed a contraction for the $\|\cdot\|_{\infty}$ norm if $T$ is small enough. The rest of the proof follows exactly the same argument as the proof of Theorem \ref{thm:lipd}.

 \end{proof}



 \section{Master equations on the set of probability measures}\label{sec:c}
\subsection{Setting and notation}
 In this section, we address master equations set on the space of probability measures in cases with and without common noise. We start by presenting the case of the master equation without common noise and we turn to the general case later on. Hence we study first the PDE
 \be\label{mfgc}
 \begin{aligned}
\partial_t &U(t,x,m) + H(x,\nabla_x U(t,x,m),m) -\sigma \Delta_x U(t,x,m)\\
&+\int_{\mathbb{R}^d}B(y,\nabla_x U(t,y,m))D_mU(t,x,m,y)m(dy)\\
&- \sigma' \int_{\mathbb{R}^d}\text{div}_y(D_mU(t,x,m,y)) dm(y) =  0 \text{ in } (0,\infty)\times \mathbb{T}^d\times \mathcal{P}(\mathbb{T}^d),
 \end{aligned}
 \ee
 \be\label{initc}
 U(0,x,m) = U_0(x,m) \text{ in } \mathbb{T}^d\times \mathcal{P}(\mathbb{T}^d).
 \ee
 Here, $\mathbb{T}^d$ is the $d$ dimensional torus, $\mptd$ is the set of probability measures on $\mathbb{T}^d$, $H: (x,p,m) \to \R$ and $B:(x,p,m) \to \R^d$ are given functions, $\sigma,\sigma' > 0$ are constants.
 \begin{Rem}
 In a lot of cases studied in the literature, $B$ is equal to $D_pH$ and $\sigma = \sigma'$. Because these assumptions play no role here, we remove them, just as it was the case in the previous sections.
 \end{Rem} 
  \begin{Rem}
The choice of the $d$ dimensional torus $\mathbb{T}^d$ as the state space does not play any particular role except the one of simplifying the formulation of some statements in the following. Moreover, the setting at hand is sometimes refers to as the one of "extended MFG" as in \citep{lions2020extended}. 
 \end{Rem}
 \begin{Rem}
 The same study could be carried on if the coefficients $\sigma$ and $\sigma'$ depend on $x$ and $m$, provided that this dependence is sufficiently smooth, but we do not consider this case to keep the following more understandable.
 \end{Rem}
 The derivatives with respect to the measure argument are defined in the following way. For a function $F: \mptd \to \R$, when it is defined, we denote for $m \in \mptd, x \in \T^d$
 \be
 \nabla_m F(m,x) = \lim_{h \to 0}\frac{F((1-h)m + h \delta_x) - F(m)}{h},
 \ee
 where $\delta_x$ is the Dirac mass at $x$. Furthermore, when it is defined, we note for $m\in \mptd, x \in \T^d$,
 \be
 D_m F(m,x) = \nabla_x \nabla_mF(m,x).
 \ee
 Let us remark that if $\nabla_m F(m,\cdot)$ and $D_mF(m,\cdot)$ are well defined, then for $m,m' \in \mptd$, $\phi : \T^d \to \R^d$
 \be
 \int_{\T^d}\nabla_mF(m,x)(m' - m)(dx) = \lim_{h \to 0}\frac{F((1-h)m + h m') - F(m)}{h},
 \ee
 \be
 \int_{\mathbb{T}^d}D_mF(m,x)\cdot\phi(x)m(dx) = \lim_{h \to 0}\frac{F((Id + h \phi)_{\#}m) - F(m)}{h},
 \ee
 where $T_{\#}m$ denotes the image measure of the mesure $m$ by the map $T$.\\

In all the following, we are going to equip $\mptd$ with the Monge-Kantorovich distance $\textbf{d}_1$ defined by
\be
\textbf{d}_1(m,m') = \sup_{f, \|\nabla_xf\|_{\infty} \leq 1} \int_{\T^d}fd(m - m').
\ee
 This distance is a metric for the weak convergence of measures. Moreover, it is a norm when extended to the set of measures on $\T^d$. Considering a function $F: \T^d \times \mptd \to \R^d$, we also introduce the notation
 \be
 \text{Lip}(F) = \sup_{x,y,m}|x-y|^{-1}|F(x,m) - F(y,m)| + \sup_{x,m,m'}\textbf{d}_1(m,m')^{-1}|F(x,m)-F(x,m')|.
 \ee
  \begin{Rem}
Having chosen this distance, adapting the following results to master equations which are set on sets of measures with different masses is quite immediate. Even though we are not going to discuss anymore this fact, we believe it is worth insisting on the fact that the choice of $\textbf{d}_1$ is particularly natural because of all the MFG in which the mass of players does not remain constant.
 \end{Rem}

 Several approaches can be taken here. We present one in details and sketch a second one later on. This first approach is hinted by the previous section which suggests to consider the PDE satisfies by the function $W(t,x,m) := \nabla_x U(t,x,m)$ where $U$ is the value function of the MFG. If the value function $U$ is indeed a (say classical) solution of \eqref{mfgc}, then its spatial gradient $W$ is a solution of
 
 \be\label{mfggrad}
 \begin{aligned}
\partial_t &W+ D_pH(x,W,m)\cdot\nabla_x W  -\sigma \Delta_x W(t,x,m)\\
&+\int_{\mathbb{R}^d}B(y,W(t,y,m),m)D_mW(t,x,m,y)m(dy)\\
&- \sigma' \int_{\mathbb{R}^d}\text{div}_y(D_mW(t,x,m,y)) dm(y) = - \nabla_x H(x,W,m) \text{ in } (0,\infty)\times \mathbb{T}^d\times \mathcal{P}(\mathbb{T}^d),
 \end{aligned}
 \ee
 The transport equation naturally associated to this previous nonlinear PDE is simply
   \be\label{lineargrad}
 \begin{aligned}
 \partial_t &V(t,x,m) -b(t,x,m)\cdot\nabla_x V(t,x,m) -\sigma \Delta_x V(t,x,m)\\
&+\int_{\mathbb{R}^d}F(t,x,m)\cdot D_mV(t,x,m,y)m(dy)\\
&- \sigma' \int_{\mathbb{R}^d}\text{div}_y(D_mV(t,x,m,y)) dm(y) = A(t,x,m) \text{ in } (0,\infty)\times \mathbb{T}^d\times \mathcal{P}(\mathbb{T}^d),
 \end{aligned}
  \ee
 where $b,B,A : [0,\infty)\times \mathbb{T}^d\times \mptd \to \R^d$ are given vector fields.\\
 
 As in the previous settings, the solutions of the transport equation \eqref{lineargrad} are naturally given by a representation formula. Moreover, from the presence of a second order term in $x$, the evolution equation associated to this variable is stochastic. This does not raise any particular mathematical difficulty. The associated Feynman-Kac representation formula is then
 \be\label{repgrad}
 \forall t \in [0,T), x \in \mathbb{T}^d, m_* \in \mptd, V(t,x,m_*) = \mathbb{E}\left[ \int_0^t A(t-s,X_s,m(s)) ds + U_0(X_t,m(t))\right],
 \ee
 where $(X_s,m(s))_{s \geq 0}$ is the solution of the following SDE-PDE system
 \begin{equation} \label{sdepdegrad}
 \begin{aligned}
 dX_s &= b(t-s,X_s,m(s))ds + \sqrt{2 \sigma}dW_s \text{ for } s\leq t,\\
 \partial_s m &= -\text{div}(F(t-s,x,m(s))m(s)) + \sigma' \Delta_x m \text{ in } (0,t)\times \mathbb{T}^d,
 \end{aligned}
 \ee
 with initial conditions $x$ and $m_*$, where $(W_s)_{s\geq 0}$ is a $d$ dimensional Brownian motion on a standard probability space $(\Omega, \mathcal{A},\mathbb{P})$. Let us recall that the link with equation \eqref{lineargrad} can be observed by computing
 \be
(dt)^{-1} V(t,x,m_*) - \mathbb{E}[V(t-dt,X_{dt},m(dt))]
 \ee
 and letting $dt\to 0$.
 
 We denote by $\Psi$ the operator which is defined by the representation formula \eqref{repgrad}. That is, the function $V$ given by the right hand side of \eqref{repgrad} is denoted by $V = \Psi(T,b,F,A,U_0)$. 
 
 \subsection{Main definition and result}
 We can now easily state the following definition of solution of \eqref{mfggrad}.
  \begin{Def}\label{def:lipgrad}
 Given an initial condition $U_0$, a Lipschitz solution of \eqref{mfggrad} on $[0,T)$ is a function $W : [0,T) \times \mathbb{T}^d\times \mptd \to \R^d$ such that
 \begin{itemize}
 \item $W$ is Lipschitz in $x,m$, uniformly in $[0,t]$ for any $t < T$.
 \item For any $t < T$, 
 \be
 W = \Psi(t, D_pH(x,W,m), B(x,W,m), - D_xH(x,W,m),U_0).
 \ee
 \end{itemize}
 \end{Def}
 \begin{Rem}
 In this context, $W$ is not asked to be a gradient in $x$.
 \end{Rem}
We can produce the same type of result for this notion of solution.
\begin{Theorem}\label{thm:grad}
Assume that:
 \begin{itemize}
 \item The function $H$ is such that $D_xH$ and $D_pH$ are globally Lipschitz functions.
 \item The function $B$ is a globally Lipschitz function.
 \end{itemize}
 Then, for any initial condition $U_0$ such that $\nabla_x U_0$ is Lipschitz:
 \begin{itemize}
\item There always exists a time $T> 0$ such that there exists a unique solution $W$ of \eqref{mfgc} in the sense of Definition \ref{def:lipgrad}.
\item There exists a maximal time $T^c \in [0,\infty]$ and a solution $W$ associated to $T^c$ such that, for any solution $V$ of the problem on an interval $[0,T]$: we have that $T \leq T^c$ and the restriction of $W$ to $[0,T)$ is equal to $V$.
\item If $T^c < \infty$, then $\text{Lip}(W(t,\cdot,\cdot)) \to \infty$ as $t \to T^c$.
  \end{itemize}
\end{Theorem}
\begin{proof}
The proof of this result follows the same line of argument as the previous ones. Consider a time $T > 0$, the map 
\be
\begin{aligned}
\Phi(W) = \Psi(T,D_pH(x,W), B(x,W,m),  - D_xH(x,W,m),U_0).
\end{aligned}
\ee
and the set 
\be
E_C := \{ W : [0,T]\times \mathbb{T}^d\times \mptd \to \R^d, \|W\|_{\infty} \leq C,\sup_{t \leq T}\text{Lip}(W(t,\cdot,\cdot)) \leq C\}.
\ee
\textbf{Step $1$: $\Phi$ is well defined.} Take $W \in E_{\infty}$. Setting $b(t,x,m) = -D_pH(x,W(t,x,m))$ and $F(t,x,m) = B(x,U(t,x,m),m)$, we remark that the system \eqref{sdepdegrad} always has strong solutions for $t < T$. Hence from the continuity of $\nabla_x U_0$ and $\nabla_x H$, we deduce that $\Phi$ is indeed well defined.\\

Let us now recall a standard estimate on Fokker-Planck equations. Consider $m_1$ and $m_2$ solutions of
\be
\begin{aligned}
\partial_s m_i - \sigma' \Delta_x m_i + \text{div}(b_i(s,x)m_i) = 0 \text{ in } (0,T)\times \mathbb{T}^d \text{ for } i =1,2,\\
m_i|_{s = 0} = µ,
\end{aligned}
\ee
where $b_1,b_2 : [0,T) \times \mathbb{T}^d \to \R^d$ are both bounded in $s,x$, and Lipschitz continuous in $x$, uniformly on $[0,t]$ for $t < T$. Then, for any $s \leq t < T$, there exists $C_t$ depending only on the bounds on $b$ such that
\be\label{estimate}
\textbf{d}_1(m_1(s),m_2(s)) \leq C \int_0^t\|b_1(s) - b_2(s)\|_{\infty} ds.
\ee
This estimate is classical and a proof is provided in appendix for the sake of completeness.\\

Let us show that $\Phi(E_{\infty}) \subset E_{\infty}$. Consider $W \in E_{\infty}$, $m_1$ and $m_2$ and let us compute for $t < T,x\in \mathbb{T}^d$,
\be\label{eq:42}
\begin{aligned}
|\Phi(W)(t,x,m_1) - \Phi(W)(t,x,m_2)| = \mathbb{E}&\bigg[\int_0^t - \nabla_x H(X_{1,s},W(t-s,X_{1,s},m_1(s)),m(s))\\
&+ \nabla_x H(X_{2,s},W(s,X_{2,s},m_2(s)),m_2(s))ds\\
&+ \nabla_x U_0(X_{1,t},m_1(t)) - \nabla_x U_0(X_{2,t},m_2(t))\bigg],
\end{aligned}
\ee
where the $(X_{i,s},m_i(s))_{s \geq 0}$ are the solutions of 
 \begin{equation}
 \begin{aligned}
 dX_{i,s} &= D_pH(X_{i,s}, W(t-s,X_{i,s},m_i(s)),m_i(s))ds + \sqrt{2 \sigma}dW_s \text{ for } s\leq t,\\
 \partial_s m_i &= -\text{div}(B(x,W(t-s,x,m_i))m_i) + \sigma' \Delta_x m_i \text{ in } (0,t)\times \mathbb{T}^d,
 \end{aligned}
 \ee
 with initial conditions $m_i$ and $X_{i,0} = x$. Furthermore, since $W \in E_{\infty}$, we can use the estimate \eqref{estimate} to deduce that, for $t\leq T$, there exists $C > 0$, depending only on the data of the problem, such that
 \be\label{estmm}
 \textbf{d}_1(m_1(t),m_2(t)) \leq e^{C\text{Lip}(W)t}\textbf{d}_1(m_1,m_2).
 \ee
 Moreover, the following holds almost surely
 \be
 \begin{aligned}
 d|X_{1,s}-X_{2,s}| \leq & \|D_{pp}H\|_{\infty}\text{Lip}(W)(\textbf{d}_1(m_1(s),m_2(s))+ |X_{1,s} - X_{2,s}| )\\
 &+\|D_{px}H\|_{\infty}|X_{1,s} - X_{2,s}|\\
 &\leq C\text{Lip}(W)|X_{1,s} - X_{2,s}| + e^{C\text{Lip}(W)T}\textbf{d}_1(m_1,m_2).
 \end{aligned}
 \ee
 Hence, we obtain using Gr\"onwall's Lemma that, almost surely, for $s \leq T$
 \be\label{estXX}
 |X_{1,s}-X_{2,s}| \leq (e^{C\text{Lip}(W)s} -1)e^{C\text{Lip}(W)T}\textbf{d}_1(m_1,m_2).
 \ee
 Using \eqref{estmm} and \eqref{estXX} in \eqref{eq:42} and using the regularity assumptions on $H$ and $f$, we finally deduce that
 \be
 |\Phi(W)(t,x,m_1) - \Phi(W)(t,x,m_2)| \leq Ce^{C\text{Lip}(W)T}\textbf{d}_1(m_1,m_2),
 \ee
where $C>0$ is a constant which depends only on $U_0,H$ and $B$. The same type of result is also true for estimating the Lipschitz constant of $\Phi(W)$ in $x$ but we do not present it here. Thus, it follows that $\Phi(E_{\infty})\subset E_{\infty}$. 

Moreover, if $C$ is sufficiently large, then for $T$ sufficiently small, $\Phi(E_C)\subset E_C$.\\

\textbf{Step $2$: $\Phi$ is a contraction.} We now show that, if $T$ is small enough, then $\Phi$ is a contraction. Take $W_1,W_2 \in E$ and compute for $t\leq T,x\in \mathbb{T}^d,m\in \mptd$
\be\label{eq:43}
\begin{aligned}
|\Phi(W_1)(t,x,m) - \Phi(W_2)(t,x,m)| = \bigg|\mathbb{E}&\bigg[\int_0^t - \nabla_x H(X_{1,s},W_1(t-s,X_{1,s},m_1(s)), m_1(s))\\
&+ \nabla_x H(X_{2,s},W_2(t-s,X_{2,s},m_2(s)), m_2(s))ds\\
& + \nabla_x U_0(X_{1,t},m_1(t))- \nabla_x U_0(X_{2,t},m_2(t))\bigg]\bigg|,
\end{aligned}
\ee
where the $(X_{i,s},m_i(s))_{s \geq 0}$ are the solutions of 
 \begin{equation}
 \begin{aligned}
 dX_{i,s} &= D_pH(X_{i,s}, W_i(t-s,X_{i,s},m_i(s)),m_i(s))ds + \sqrt{2 \sigma}dW_s \text{ for } s\leq t,\\
 \partial_s m_i &= -\text{div}(B(x,W_i(t-s,x,m_i),m_i)m_i) + \sigma' \Delta_x m_i \text{ in } (0,t)\times \mathbb{T}^d,
 \end{aligned}
 \ee
 with initial conditions $m_i= m$ and $X_{i,0} = x$. From estimate \eqref{estimate} and the regularity of $B$, we obtain that for $t\leq T$, there exists $C> 0$ such that 
 \be
 \begin{aligned}
 \textbf{d}_1(m_1(t),m_2(t)) &\leq C \int_0^t\|W_1(t-s,\cdot,m_1(s)) - W_2(t-s,\cdot,m_2(s))\|_{\infty}ds\\
 &\leq C \int_0^t \text{Lip}(W_1)\textbf{d}_1(m_1(s),m_2(s)) + \|W_1(t-s,\cdot,\cdot) - W_2(t-s,\cdot,\cdot)\|_{\infty}ds.
 \end{aligned}
 \ee
 Using once again Gr\"onwall's Lemma, we obtain that
 \be
 \textbf{d}_1(m_1(t),m_2(t)) \leq C(e^{C\text{Lip}(W_1)t} -1)\|W_1 -W_2\|_{\infty}.
 \ee
 Hence it follows that, for any $\alpha \in (0,1)$, if $T$ is chosen small enough (where $C$ was already chosen sufficiently large), for all $W_1,W_2 \in E_C$
 \be
  \textbf{d}_1(m_1(t),m_2(t)) \leq \alpha \|W_1 - W_2 \|_{\infty}.
 \ee
 Using the same type of argument as in \textbf{Step $1$}, we can obtain the same estimate (almost surely) on $|X_{1,t} - X_{2,t}|$. From this, recalling \eqref{eq:43}, we finally obtain that
 \be
 \|\Phi(W_1)-\Phi(W_2)\|_{\infty} \leq \alpha \|W_1 - W_2 \|_{\infty}.
 \ee
 Hence $\Phi$ is a contraction from $(E_C,\|\cdot\|_{\infty})$ into itself.\\
 
 \textbf{Step $3$: Existence of a fixed point.} From the previous step, we deduce that for $W \in E_C$, $(\Phi^n(W))_{n \geq 0}$ is a Cauchy sequence. Moreover it is valued in $E_C$ and $(\|\partial_t\Phi^n(W)\|_{\infty})_{n \geq 1}$ is a bounded sequence. From Ascoli-Arzela Theorem, we deduce that $(\Phi^n(W))_{n \geq 0}$ converges uniformly to some $W_* \in E_C$. From the continuity of $\Phi$, $\Phi(W_*) = W_*$.\\
 
The rest of the result follows quite easily from standard arguments.
\end{proof}
Moreover, we also have the
\begin{Cor}\label{cor:nonlocal}
The result of the Theorem remains true if the dependence on $W$ in $H$ and $B$ is non local, as long as the Lipschitz regularity holds for the $\|\cdot\|_{\infty}$ norm.
\end{Cor}
\begin{proof}
It suffices to follow the previous proof and remark that we never used explicitly that the dependence was local, and that we always used the $\|\cdot\|_{\infty}$ norm anyway.
\end{proof}
 
 \subsection{Return to the solutions of the initial master equation}
 To complete the study of \eqref{mfgc}, we now explain how we can use the knowledge of the solutions of \eqref{mfggrad} to define solutions of \eqref{mfgc}.\\
 
 On the time interval $[0,T)$ for $T > 0$, if $W$ is the unique Lipschitz solution of \eqref{mfggrad} and $U$ is the unique solution of \eqref{mfgc}, then we expect that $\nabla_x U = W$. Hence we expect that $U$ is a solution of
 \be
  \begin{aligned}
\partial_t &U + H(x,W(t,x,m),m) -\sigma \Delta_x U(t,x,m)\\
&+\int_{\mathbb{R}^d}B(y,W(t,y,m),m)D_mU(t,x,m,y)m(dy)\\
&- \sigma' \int_{\mathbb{R}^d}\text{div}_y(D_mU(t,x,m,y)) dm(y) = 0 \text{ in } (0,\infty)\times \mathbb{T}^d\times \mathcal{P}(\mathbb{T}^d).
 \end{aligned}
 \ee
 The previous PDE is a linear transport equation in $U$. Hence it can be dealt with by means of the following representation formula, for $t < T, x \in \mptd, µ \in \mptd$
 \be\label{charc}
 U(t,x,µ) = \mathbb{E}\left[\int_0^t - H(X_s,W(t-s,X_s,m(s)))ds + U_0(X_t,m(t)) \right],
 \ee
 where $(X_s,m(s))_{s \in [0,t]}$ is the solution of 
 \be
  \begin{aligned}
 dX_s &=  \sqrt{2 \sigma}dW_s \text{ for } s\in (0,t),\\
 \partial_t m &= -\text{div}(B(x,W(t-s,x,m))m) + \sigma' \Delta_x m \text{ in } (0,t)\times \mathbb{T}^d,
 \end{aligned}
 \ee
 with initial conditions $ x$ and $ µ$. This leads us to the definition
 \begin{Def}
 A bounded function $U : [0,T)\times \mathbb{T}^d\times \mptd \to \R$ is a Lipschitz solution of \eqref{mfgc} on the time interval $[0,T)$ if there exists $W$, Lipschitz solution of \eqref{mfggrad} on $[0,T)$ such that $U$ satisfies \eqref{charc}.
 \end{Def}
 \begin{Rem}
 Let us insist that we do not need to impose the facts that $W$ satisfies $W = \nabla_x U$ nor that it is a gradient. Somehow, we translate here the fact that, given the controls of the players, i.e. the function $W$, we can simply compute the value by following the characteristics.
 \end{Rem}
  \begin{Rem}
 This Definition makes clear that, in this setting, the appropriate Lipschitz regularity is on $W=\nabla_x U$ and not on $U$. This type of fact is often interpreted in the literature as the fact that \eqref{mfgc} is an equation on the controls of the players as well as on the value.
 \end{Rem}
 As a consequence of Theorem \ref{thm:grad}, we obtain the
 \begin{Theorem}
 Under the assumptions of Theorem \ref{thm:grad}:
 \begin{itemize}
 \item There exists $T > 0$ such that there is a Lipschitz solution $U$ of \eqref{mfgc} on $[0,T)$. 
 \item There is a maximal time of existence $T^c \in (0,\infty]$ and it is such that any Lipschitz solution $V$ of \eqref{mfgc} on a time interval $[0,T)$ is such that $T\leq T^c$ and the restriction of $U$ to $[0,T)$ is equal to $V$.
 \end{itemize}
 \end{Theorem}
 \begin{proof}
 This result is a direct application of Theorem \ref{thm:grad}.
 \end{proof}
 \begin{Rem}
 Another notion of solution of \eqref{mfgc} could have been introduced, based on the existence of a solution of \eqref{mfggrad}. Indeed we could have replaced all the terms involving $U$ in \eqref{mfgc}, except $\partial_t U$, by terms involving $W$. It then suffices to check that all the terms in $W$ are bounded and thus that this new equation characterizes $\partial_t U$. However, we did not use this route since the one we choose allows us to avoid technical problems such as giving a precise sense to the term $\nabla_m W$, which should appear with this new method.
 \end{Rem}

 \subsection{Master equations associated to common noise}
 As in the Hilbertian case, the notion of Lipschitz solutions can easily be adapted to cases involving so-called common noises. Let us recall that the presence of a common noise (which should probably be called common shocks) in a MFG usually translates into the addition of terms in the master equation, which can be either non-local or of higher order. We refer to \citep{bertucci2021monotone2} for several examples of possible common noises. Even though the previous approach seems quite general and should work in all the cases, we focus here on the following master equation.
 
 \be\label{mfgcc}
 \begin{aligned}
\partial_t &U(t,x,m) + H(x,\nabla_x U(t,x,m),m) -(\sigma +\sigma_0) \Delta_x U(t,x,m)\\
&+\int_{\mathbb{R}^d}B(y,\nabla_x U(t,y,m),m)D_mU(t,x,m,y)m(dy)\\
&- (\sigma' + \sigma_0) \int_{\mathbb{R}^d}\text{div}_y(D_mU(t,x,m,y)) dm(y) - 2 \sigma_0 \int_{\mathbb{T}^d}\text{div}_x(D_mU(t,x,m,y))m(dy)\\
& - \sigma_0 \int_{\mathbb{T}^{2d}}Tr[D^2_{mm}U(t,x,m,y,z)]m(dy)m(dz) = 0 \text{ in } (0,\infty)\times \mathbb{T}^d\times \mathcal{P}(\mathbb{T}^d),
 \end{aligned}
 \ee
 with initial condition $U_0 : \mathbb{T}^d\times \mptd \to \R$. A study similar to the one we just conducted can be done here. Indeed, we can consider first the equation satisfied by $W =\nabla_x U$. In this context, this PDE is
  \be\label{mfggradc}
 \begin{aligned}
\partial_t &W(t,x,m) + D_p H(x,W(t,x,m),m)\cdot \nabla_x W(t,x,m) -(\sigma +\sigma_0) \Delta_x W(t,x,m)\\
&+\int_{\mathbb{R}^d}B(y,W(t,y,m),m)D_mW(t,x,m,y)m(dy) + \nabla_x H(x,W(t,x,m),m) \\
&- (\sigma' + \sigma_0) \int_{\mathbb{R}^d}\text{div}_y(D_mW(t,x,m,y)) dm(y) - 2 \sigma_0 \int_{\mathbb{T}^d}\text{div}_x(D_mW(t,x,m,y))m(dy)\\
& - \sigma_0 \int_{\mathbb{T}^{2d}}Tr[D^2_{mm}W(t,x,m,y,z)]m(dy)m(dz) = 0 \text{ in } (0,\infty)\times \mathbb{T}^d\times \mathcal{P}(\mathbb{T}^d).
 \end{aligned}
 \ee
 Furthermore, associated to this equation, we can also consider the linear transport equation
   \be\label{lineargradc}
 \begin{aligned}
\partial_t &V(t,x,m) - b(t,x,m)\cdot \nabla_x V(t,x,m) -(\sigma +\sigma_0) \Delta_x V(t,x,m)\\
&+\int_{\mathbb{R}^d}F(t,y,m)D_mV(t,x,m,y)m(dy)\\
&- (\sigma' + \sigma_0) \int_{\mathbb{R}^d}\text{div}_y(D_mV(t,x,m,y)) dm(y) - 2 \sigma_0 \int_{\mathbb{T}^d}\text{div}_x(D_mV(t,x,m,y))m(dy)\\
& - \sigma_0 \int_{\mathbb{T}^{2d}}Tr[D^2_{mm}V(t,x,m,y,z)]m(dy)m(dz) = A(t,x,m) \text{ in } (0,\infty)\times \mathbb{T}^d\times \mathcal{P}(\mathbb{T}^d),
 \end{aligned}
 \ee
 which is associated to the system of SDE-stochastic PDE
 \be\label{sdespde}
 \begin{aligned}
 dX_s &= b(t-s,X_s,m_s)ds + \sqrt{2\sigma}dW_s + \sqrt{2\sigma_0}dW'_s \text{ for } s \in (0,t),\\
 dm_s &= [(\sigma' + \sigma_0)\Delta_x m -\text{div}(F(t-s,x,m_s)m_s)]ds -\text{div}(m_s \sqrt{2\sigma_0}dW'_s) \text{ in } (0,t)\times \mathbb{T}^d,
 \end{aligned}
 \ee
 where $(W_t)_{t \geq 0}$ and $(W'_t)_{t\geq 0}$ are two independent Brownian motions on the standard probability space $(\Omega, \mathcal{A},\mathbb{P})$. Note that the fact that $(W')_{t\geq 0}$ appears in the two equations is fundamental to obtain the crossed derivatives term which is the term in $2\sigma_0$ in \eqref{mfgcc}.
 
  Let us insist on the fact that the previous system has a unique strong solution as soon as $b$ and $F$ are Lipschitz in $x,m$, uniformly in time. Indeed, the additional Brownian motion merely acts as a translation here and does not perturb too much the mathematical analysis. As we did several times above, we can associate a representation formula to \eqref{lineargradc} which reads for $t\geq 0, x \in \mathbb{T}^d, µ \in \mptd$
 \be
 V(t,x,µ) = \mathbb{E}\left[ \int_0^t A(t-s,X_s,m_s)ds + U_0(X_t,m_t)\right],
 \ee
 where $(X_s,m_s)_{s \in [0,t]}$ is the unique (strong) solution of \eqref{sdespde} with initial conditions $x$ and $µ$. Denoting by $\Psi$ the operator defined by this formula, we can introduce the
 \begin{Def}\label{def:lipgradc}
 Given $T > 0$, a Lipschitz solution of \eqref{mfggradc} on $[0,T)$ is a function $W : [0,T)\times \mathbb{T}^d\times \mptd \to \R^d$ such that 
 \begin{itemize}
 \item $W$ is Lipschitz in $x,m$, uniformly in $t \in [0,\alpha]$ for $\alpha < T$.
 \item The following holds for any $t < T$
 \be
 W = \Psi(T,- \nabla_x H(x,W),D_pH(x,W,m),B(x,W,m),U_0).
 \ee
 \end{itemize}
 \end{Def}
 We can establish the following result
 \begin{Theorem}
 Under the assumptions of Theorem \ref{thm:grad}:
 \begin{itemize}
 \item There always exists a time $T > 0$ such that there exists a unique solution $W$ of \eqref{mfggradc} on $[0,T)$ in the sense of Definition \ref{def:lipgradc}.
 \item There exists a maximal time $T^c\in [0,\infty]$ and a solution $W$ associated to $T^c$ such that, for any solution $V$ of the problem on an interval $[0,T)$: we have that $T \leq T^c$ and the restriction of $W$ to $[0,T)$ is equal to $V$.
 \item If $T^c< \infty$, then $\text{Lip}(W(t,\cdot,\cdot)) \to \infty$ as $t \to T^c$.
 \end{itemize}
 \end{Theorem}
The previous statement is word for word the same as the one in the case $\sigma' = 0$ and the same almost holds for their proofs. That is why we only sketch the proof here, mainly by highlighting the main differences with the proof of Theorem \ref{thm:grad}.
 \begin{proof}
 We use the same notation as in the proof of Theorem \ref{thm:grad}.\\
 
 The only key argument here consists in showing that the estimate \eqref{estimate} can also be used in this stochastic case. Consider $m_1$ and $m_2$, two solutions of
 \be
  dm_{i,s} = [(\sigma' + \sigma_0)\Delta_x m_{i,s} -\text{div}(b_i(s,x)m_{i,s})]ds -\text{div}(m_{i,s} \sqrt{2\sigma_0}dW'_s) \text{ in } (0,t)\times \mathbb{T}^d,
 \ee
 for $(W'_t)_{t\geq 0}$ a standard Brownian motion and $b_1,b_2: [0,\infty)\times\mathbb{T}^d\times \mptd\to \R^d$ two bounded vector fields. Consider $\tilde{m}_1$ and $\tilde{m}_2$ given by 
 \be
 \tilde{m}_{i,s} = (\tau_{\sqrt{2\sigma'}W'_s})_{\#}m_{i,s},
 \ee
 where $\tau_x: \T^d \to \T^d$ is the translation of $x$ and $T_{\#}µ$ denotes the image measure of the measure $µ$ by the map $T$. Let us remark that, for any $s \geq 0$
 \be
 \textbf{d}_1(m_{1,s},m_{2,s}) =  \textbf{d}_1(\tilde{m}_{1,s},\tilde{m}_{2,s}).
 \ee
 We now observe that, by construction, for any $\omega \in \Omega$,
 \be
 \partial_t \tilde{m}_{i} -\sigma' \Delta \tilde{m}_i + \text{div}(\tilde{b}_i\tilde{m}_i) = 0 \text{ in } (0,t)\times \T^d \times \mptd,
 \ee
 where $\tilde{b}_i(s,x) := b_i(s,x+\sqrt{2\sigma_0}W'_s)$. Remarking finally that
 \be
 \|\tilde{b}_1 - \tilde{b}_2\|_{\infty} = \|b_1 - b_2\|_{\infty},
 \ee
 we deduce that the estimate \eqref{estimate} is satisfied almost surely (with a constant independent of $\omega \in \Omega$) in this stochastic case.\\
 
 The rest of the proof follows the same argument as in the case without common noise.
 \end{proof}
 As in the case without common noise, we can of course use a notion of solution of \eqref{mfggradc} to establish a definition and results on solutions of \eqref{mfgcc}. Although we do not detail it here as it will merely be a copy of the previous case.
  
  \subsection{Master equations involving the image measure}
  In several MFG models, see for instance \citep{carmona2018probabilistic,cardaliaguet2016mean,bertucci2019some,kobeissi2022classical}, the dependence on $m$ of the non-linearities $D_pH$ and $B$ in \eqref{mfgc} happens through the image measure of $m$ by a certain function of the gradient in $x$ of the value function. Hence, in such cases, the non-linearities have a form similar to $\tilde{A} : \T^d \times (\T^d\to\R^d)\times \mptd \to \R^d$
  \be
 \tilde{A}(x,\phi,m):= A(x,m,\phi(x), \psi(\phi)_{\#}m).
  \ee
  where $A: \T^d\times\mptd\times\R^d\times \mptd\to \R^d$, $\psi : \R^d \to \R^d$ and $T_{\#}m$ denotes the image measure of $m$ by the map $T$. The following result holds.
 
\begin{Prop}
If $A$ is Lipschitz, then so is $\tilde{A}$ on $\T^d\times E_C\times \mptd$, where $E_C$ is the set of Lipschitz functions $\T^d \to \R^d$ with Lipschitz constant at most $C$.
\end{Prop}
\begin{proof}
Let us first compute, for $\phi, \phi'$ two Lipschitz functions $\T^d \to \R^d$ with Lipschitz constant $C$ and $µ,µ' \in \mptd$
\be
\begin{aligned}
\textbf{d}_1(\psi(\phi)_{\#}µ,\psi(\phi')_{\#}µ') &= \sup_{\|f\|_{Lip} \leq 1} \left\{ \int_{\T^d}f(\psi(\phi(x)))µ(dx) - \int_{\T^d}f(\psi(\phi'(x)))µ'(dx)\right\}\\
&\leq \sup_{\|f\|_{Lip} \leq 1} \left\{ \int_{\T^d}f(\psi(\phi(x)))µ(dx) - \int_{\T^d}f(\psi(\phi(x)))µ'(dx)\right\}\\
&\quad + \sup_{\|f\|_{Lip} \leq 1} \left\{ \int_{\T^d}f(\psi(\phi(x)))-f(\psi(\phi'(x)))µ'(dx)\right\}\\
& \leq C\|\psi\|_{Lip}\textbf{d}_1(µ,µ') + \|\psi\|_{Lip}\|\phi - \phi'\|_{\infty}.
\end{aligned}
\ee
Hence we deduce that $\tilde{A}$ is Lipschitz continuous
\end{proof}
From this property, we easily deduce from Corollary \ref{cor:nonlocal} that Theorem \ref{thm:grad} can be extended to situations involving the image measure, that is master equation of the form
 \be\label{mfgim}
 \begin{aligned}
\partial_t &U(t,x,m) + H(x,\nabla_x U(t,x,m),m,\psi(\nabla_xU)_{\#}m) -\sigma \Delta_x U(t,x,m)\\
&+\int_{\mathbb{R}^d}B(y,\nabla_x U(t,y,m),m,\psi'(\nabla_xU)_{\#}m)D_mU(t,x,m,y)m(dy)\\
&- \sigma' \int_{\mathbb{R}^d}\text{div}_y(D_mU(t,x,m,y)) dm(y) =  0 \text{ in } (0,\infty)\times \mathbb{T}^d\times \mathcal{P}(\mathbb{T}^d),
 \end{aligned}
 \ee
 for $\psi$ and $\psi'$ two Lipschitz functions.

  \section{Comments and future developments on Lipschitz solutions of MFG master equations}\label{sec:add}
  We present in this sections several comments and future directions of research on Lipschitz solutions of MFG master equations that we believe could be of interest.
  \subsection{Application to numerical computations}
  Let us precise what we believe to be the principal application of the previous uniqueness result : the justification of numerical computations. Indeed, if by using some abstract or black-box method, like neural networks for instance, one is able to exhibit a solution of a MFG master equation of one of the types presented above, then if it is a Lipschitz solution, it is necessary the unique one.\\
  
  With the growing number of works on the use of machine learning techniques to solve MFG master equation, and the lack of proof of convergence results, the results presented above justify the following heuristic : if one Lipschitz solution has been selected, then it have at least some meaning in the sense that it is the only one. Moreover, let us insist that, because several machine learning methods are parametrized (neural networks for instance). Hence establishing some regularity properties of the learned solution can be done a priori.
  
  \subsection{Another representation formula}
  
 We provide here another approach to represent solutions of \eqref{mfgc}, in which we linearize only the transport term in $m$. To be more precise, we consider the following equation
  \be\label{mfgsemi}
 \begin{aligned}
\partial_t &V(t,x,m) + H(x,\nabla_x V(t,x,m),m) -\sigma \Delta_x V(t,x,m)\\
&+\int_{\mathbb{R}^d}B(t,y,m)\cdot D_mV(t,x,m,y)m(dy)\\
&- \sigma' \int_{\mathbb{R}^d}\text{div}_y(D_mV(t,x,m,y)) dm(y) = f(x,m) \text{ in } (0,\infty)\times \mathbb{T}^d\times \mathcal{P}(\mathbb{T}^d),
 \end{aligned}
 \ee
 where $B : [0,\infty)\times \mathbb{T}^d\times \mptd \to \R$ is a vector field. To this equation, we naturally associates the system
 \be\label{sdepde2}
 \begin{aligned}
 dX_s &= \alpha_s ds + \sqrt{2 \sigma}dW_s \text{ for } t\geq 0,\\
 \partial_s m &= -\text{div}(B(t-s,x,m)m) + \sigma' \Delta_x m \text{ in } (0,t)\times \mathbb{T}^d,
 \end{aligned}
 \ee
 where $(\alpha_s)_{s \geq 0}$ represents the control of a player. In this framework, the natural representation of a solution of \eqref{mfgsemi} is 
 \be\label{charsemi}
V(t,x,µ) = \inf_{\alpha}\mathbb{E}\left[ \int_0^t  H^*(X^{\alpha}_s,\alpha_s,m(s)) ds + U_0(X^{\alpha}_t,m(t))\right],
 \ee
 where $(X^{\alpha}_s,m(s))_{s \in [0,t]}$ is the solution of \eqref{sdepde2} with initial conditions $x$ and $µ$ and where $H^*$ is the Fenchel conjugate of $H$ with respect to its second argument. If the vector field $B$ is Lipschitz continuous, it is a standard result of stochastic optimal control that $V$ is indeed well defined. As in the previous case we denote by $\Psi$ the operator defined by the previous relation. That is, if $V$ is given by \eqref{charsemi}, then we denote 
 \be
 V = \Psi(t,B,U_0).
 \ee
 We can introduce the following notion of solution.
  \begin{Def}\label{def:lipc}
 Given a time $T$ and an initial condition $U_0$, a Lipschitz solution of \eqref{mfgc} is a function $U : [0,T) \times \mathbb{T}^d\times \mptd \to \R$ such that
 \begin{itemize}
 \item $\nabla_xU$ is Lipschitz in $x,m$, uniformly in $[0,t]$ for any $t < T$.
 \item For any $t < T$, 
 \be\label{fixedpoint}
 U = \Psi(t, D_pH(x,\nabla_x U,m),U_0).
 \ee
 \end{itemize}
 \end{Def}
We do not provide a detailed mathematical study of this notion of solution, mainly because it will be redundant with the one we provided in Section \ref{sec:c} and also more technical. Nonetheless, this notion of solution is easily interpretable in terms of game theory. Indeed, in this context, $\Psi(t,B,U_0)$ is simply the operator which computes the value associated to the best response strategy for a generic player, given that it anticipates the vector field $B$. By anticipation of the vector field $B$, we mean that it anticipates that the repartition of players is going to be driven by $B$. Hence, a fixed point of \eqref{fixedpoint} is easily interpretable in terms of a fixed point of a best reply operator. 

Finally, this operator $\Psi$ is, in general in game theory contexts, of a practical use outside of just defining a notion of equilibria for games. It can help to study procedures such as fictitious play for instance.

\subsection{Another notion of monotone solutions for MFG master equations}
The notion of Lipschitz solution we just presented is based on the idea that through some Lipschitz regularity, we can define characteristics for the equation and then a value through a representation formula. Another natural regime in which the flow of an evolution equation is well defined is the monotone one, see for instance \citep{brezis}. We now make a brief development on how we could use this idea to define a solution of the master equation, in the same spirit as what we did for the Lipschitz regularity. We work in the case of \eqref{mfgH} to fix ideas.\\

In the setting of \eqref{mfgHn}, we recall that the natural characteristics of the master equation have the form
\be\label{ode4}
\frac{d X(s)}{ds} = -F(X(s),U(t-s,X(s))),
\ee
with initial condition
\be
X(0) = x_0 \in H.
\ee
The previous ODE admits a unique solution as soon as $(t,X) \to F(X,U(t,X))$ is Lipschitz in $X$, uniformly in $t$. But is also the case when 
\be\label{mon}
\forall t \geq 0, X,Y \in H, \langle F(X,U(t,X))- F(Y,U(t,Y)), X-Y \rangle \geq 0.
\ee
Based on this remark, we can provide the following definition
\begin{Def}
For $T > 0$, given an initial condition $U_0$, a bounded function $U:[0,T]\times H \to H$ is a solution of \eqref{mfgH} on $[0,T]$ if
\begin{itemize}
\item For $t\leq T$, $U$ satisfies \eqref{mon}.
\item For $t \leq T, x \in H$, $U$ satisfies
\be
U(t,x) = \int_0^t G(X(s),U(t-s,X(s)))ds + U_0(X(t)),
\ee
where $(X(s))_{s \in [0,t]}$ is the unique solution of \eqref{ode4} with initial condition $x$.
\end{itemize}
\end{Def}
Obviously, this notion of monotonicity is different from the usual one in MFG theory and one does not clearly imply the other.

\subsection{Master equations with control on the volatility}
In the setting of Section \ref{sec:c}, we can also consider master equations which are associated to a MFG in which the players control the volatility of their trajectory. In such a setting, the master equations takes the form of 
\be
\begin{aligned}
\partial_t &U(t,x,m) +\int_{\mathbb{T}^d}Tr[G(y,D^2_x U(t,y,m),m)D^2_y\nabla_mU(t,x,m,y)]m(dy) \\
&+ F(x,D^2_x U(t,x,m),m)=  0 \text{ in } (0,\infty)\times \mathbb{T}^d\times \mathcal{P}(\mathbb{T}^d),
 \end{aligned}
\ee
 where $F:\T^d \times S_d(\R)\times \mptd \to \R$ and $G :\T^d \times S_d(\R)\times \mptd \to S_d(\R)$ are given functions. Recall that $D^2_x \phi$ stands for the Hessian matrix of a function $\phi: \T^d\to \R$.\\
 
Previously, we developed a theory of Lipschitz solution of the master equation based on the equation satisfied by $W = \nabla_x U$. In this context, we believe that a similar approach can be developed based on properties of the equation satisfied by $W = D^2_x U$. However, a direct use of the argument of the previous proofs is not possible and new results are here needed.

\section*{Acknowledgments} The three authors acknowledge a partial support from the Lagrange Mathematics and Computing Research Center and a partial support from the chair FDD (Institut Louis Bachelier).

\bibliographystyle{plainnat}
\bibliography{bibremarks}

\begin{thebibliography}{22}
\providecommand{\natexlab}[1]{#1}
\providecommand{\url}[1]{\texttt{#1}}
\expandafter\ifx\csname urlstyle\endcsname\relax
  \providecommand{\doi}[1]{doi: #1}\else
  \providecommand{\doi}{doi: \begingroup \urlstyle{rm}\Url}\fi

\bibitem[Achdou et~al.(2022)Achdou, Bertucci, Lasry, Lions, Rostand, and
  Scheinkman]{achdou2022class}
Yves Achdou, Charles Bertucci, Jean-Michel Lasry, Pierre-Louis Lions, Antoine
  Rostand, and Jos{\'e}~A Scheinkman.
\newblock A class of short-term models for the oil industry that accounts for
  speculative oil storage.
\newblock \emph{Finance and Stochastics}, pages 1--39, 2022.

\bibitem[Ambrose and M{\'e}sz{\'a}ros(2021)]{ambrose2021well}
David~M Ambrose and Alp{\'a}r~R M{\'e}sz{\'a}ros.
\newblock Well-posedness of mean field games master equations involving
  non-separable local hamiltonians.
\newblock \emph{arXiv preprint arXiv:2105.03926}, 2021.

\bibitem[Bertucci(2021{\natexlab{a}})]{bertucci2021monotone}
Charles Bertucci.
\newblock Monotone solutions for mean field games master equations: finite
  state space and optimal stopping.
\newblock \emph{Journal de l{\textquoteright}\'Ecole polytechnique
  {\textemdash} Math\'ematiques}, 8:\penalty0 1099--1132, 2021{\natexlab{a}}.

\bibitem[Bertucci(2021{\natexlab{b}})]{bertucci2021monotone2}
Charles Bertucci.
\newblock Monotone solutions for mean field games master equations: continuous
  state space and common noise.
\newblock \emph{arXiv preprint arXiv:2107.09531}, 2021{\natexlab{b}}.

\bibitem[Bertucci et~al.(2019)Bertucci, Lasry, and Lions]{bertucci2019some}
Charles Bertucci, Jean-Michel Lasry, and Pierre-Louis Lions.
\newblock Some remarks on mean field games.
\newblock \emph{Communications in Partial Differential Equations}, 44\penalty0
  (3):\penalty0 205--227, 2019.

\bibitem[Bertucci et~al.(2020)Bertucci, Bertucci, Lasry, and
  Lions]{bertucci2020mean}
Charles Bertucci, Louis Bertucci, Jean-Michel Lasry, and Pierre-Louis Lions.
\newblock Mean field game approach to bitcoin mining.
\newblock \emph{arXiv preprint arXiv:2004.08167}, 2020.

\bibitem[Bertucci et~al.(2021)Bertucci, Lasry, and Lions]{bertucci2021master}
Charles Bertucci, Jean-Michel Lasry, and Pierre-Louis Lions.
\newblock Master equation for the finite state space planning problem.
\newblock \emph{Archive for Rational Mechanics and Analysis}, 242\penalty0
  (1):\penalty0 327--342, 2021.

\bibitem[Brezis(1973)]{brezis}
Haim Brezis.
\newblock \emph{Op\'erateurs maximaux monotones et semi-groupes de contractions
  dans les espaces de Hilbert}.
\newblock Elsevier, 1973.

\bibitem[Cardaliaguet and Lehalle(2017)]{cardaliaguet2016mean}
Pierre Cardaliaguet and Charles-Albert Lehalle.
\newblock Mean field game of controls and an application to trade crowding.
\newblock \emph{Mathematics and Financial Economics}, pages 1--29, 2017.

\bibitem[Cardaliaguet and Souganidis(2022)]{cardaliaguet2022monotone}
Pierre Cardaliaguet and Panagiotis Souganidis.
\newblock Monotone solutions of the master equation for mean field games with
  idiosyncratic noise.
\newblock \emph{SIAM Journal on Mathematical Analysis}, 54\penalty0
  (4):\penalty0 4198--4237, 2022.

\bibitem[Cardaliaguet et~al.(2019)Cardaliaguet, Delarue, Lasry, and
  Lions]{cardaliaguet2019master}
Pierre Cardaliaguet, Fran{\c{c}}ois Delarue, Jean-Michel Lasry, and
  Pierre-Louis Lions.
\newblock \emph{The Master Equation and the Convergence Problem in Mean Field
  Games:(AMS-201)}, volume 201.
\newblock Princeton University Press, 2019.

\bibitem[Cardaliaguet et~al.(2022)Cardaliaguet, Cirant, and
  Porretta]{cardaliaguet2022splitting}
Pierre Cardaliaguet, Marco Cirant, and Alessio Porretta.
\newblock Splitting methods and short time existence for the master equations
  in mean field games.
\newblock \emph{Journal of the European Mathematical Society}, 2022.

\bibitem[Carmona and Delarue(2018{\natexlab{a}})]{carmona2018probabilistic}
Ren\'e Carmona and Fran{\c{c}}ois Delarue.
\newblock Probabilistic theory of mean field games: vol. i, mean field fbsdes,
  control, and games.
\newblock \emph{Springer}, 2018{\natexlab{a}}.

\bibitem[Carmona and Delarue(2018{\natexlab{b}})]{carmona2018probabilistic2}
Ren\'e Carmona and Fran{\c{c}}ois Delarue.
\newblock Probabilistic theory of mean field games: vol. ii, mean field games
  with common noise and master equations.
\newblock \emph{Springer}, 2018{\natexlab{b}}.

\bibitem[Gangbo and M{\'e}sz{\'a}ros(2020)]{gangbo}
Wilfrid Gangbo and Alp{\'a}r~R M{\'e}sz{\'a}ros.
\newblock Global well-posedness of master equations for deterministic
  displacement convex potential mean field games.
\newblock \emph{arXiv preprint arXiv:2004.01660}, 2020.

\bibitem[Gangbo et~al.(2021)Gangbo, M{\'e}sz{\'a}ros, Mou, and Zhang]{zhang}
Wilfrid Gangbo, Alp{\'a}r~R M{\'e}sz{\'a}ros, Chenchen Mou, and Jianfeng Zhang.
\newblock Mean field games master equations with non-separable hamiltonians and
  displacement monotonicity.
\newblock \emph{arXiv preprint arXiv:2101.12362}, 2021.

\bibitem[Kobeissi(2022)]{kobeissi2022classical}
Ziad Kobeissi.
\newblock On classical solutions to the mean field game system of controls.
\newblock \emph{Communications in Partial Differential Equations}, 47\penalty0
  (3):\penalty0 453--488, 2022.

\bibitem[Lasry and Lions(2007)]{lasry2007mean}
Jean-Michel Lasry and Pierre-Louis Lions.
\newblock Mean field games.
\newblock \emph{Japanese Journal of Mathematics}, 2\penalty0 (1):\penalty0
  229--260, 2007.

\bibitem[Lions(2006-2012)]{lions2007cours}
Pierre-Louis Lions.
\newblock Cours au college de france.
\newblock \emph{www.college-de-france.fr}, 2006-2012.

\bibitem[Lions(2021-2022)]{lions2022cours}
Pierre-Louis Lions.
\newblock Cours au college de france.
\newblock \emph{www.college-de-france.fr}, 2021-2022.

\bibitem[Lions and Souganidis(2020)]{lions2020extended}
Pierre-Louis Lions and Panagiotis~E Souganidis.
\newblock Extended mean-field games.
\newblock \emph{Rendiconti Lincei}, 31\penalty0 (3):\penalty0 611--625, 2020.

\bibitem[Mou and Zhang(2020)]{mou}
Chenchen Mou and Jianfeng Zhang.
\newblock Wellposedness of second order master equations for mean field games
  with nonsmooth data.
\newblock \emph{arXiv preprint arXiv:1903.09907}, 2020.

\end{thebibliography}

\appendix
\section{Proof of the estimate used in the proof of Theorem \ref{thm:grad}}
Remark that $µ =m_1-m_2$ is a solution of 
\be
\partial_t µ - \sigma \Delta µ + \text{div}(bµ) = \text{div}((b_2-b_1)m_2) \text{ in } (0,t_1) \times \mathbb{T}^d.
\ee
Consider now $t<t_1$ and a Lipschitz function $\phi_0 : \mathbb{T}^d \to \R$ and consider the solution $\phi$ of 
\be\label{eq:annexe}
\begin{cases}
-\partial_s \phi - \sigma \Delta \phi - b \cdot \nabla_x \phi = 0 \text{ in } (0,t)\times \mathbb{T}^d,\\
\phi|_{s = t} = \phi_0 \text{ in } \mathbb{T}^d.
\end{cases}
\ee
Using the fact that $µ$ is in particular a weak solution of the previous PDE, we deduce that
\be
\int_{\mathbb{T}^d}\phi_0 dµ_t = \int_0^t\int_{\mathbb{T}^d}\nabla_x \phi(s,x) \cdot(b_1 - b_2)(s,x)m_2(s)(dx).
\ee
Then result then follows from standard parabolic estimates on \eqref{eq:annexe}.

\end{document}